\documentclass{article}
\usepackage{amsmath,amssymb,amsthm}
\usepackage[hidelinks]{hyperref}
\usepackage{yfonts}
\usepackage{enumerate}
\usepackage{color}
\usepackage{mathtools}
\usepackage{thmtools,thm-restate}
\usepackage{tikz}
\usepackage[justification=centering]{caption}
\usepackage{ wasysym } 
\usepackage{verbatim}
\usepackage{bbm}

\usepackage{geometry}

\usepackage{graphicx} 

\newtheorem{theorem}{Theorem}[section]
\newtheorem{lemma}[theorem]{Lemma}

\newtheorem{claim}[theorem]{Claim}

\title{Finding a good tree to burn}
\author{Anders Martinsson\thanks{Department of Computer Science, ETH Z\"urich, Switzerland\newline anders.martinsson@inf.ethz.ch}}
\date{\today}

\begin{document}

\maketitle
\begin{abstract}
    The burning number of a graph $G$ is the smallest positive integer $k$ such that the vertex set of $G$ can be covered with balls of radii $0, 1, \dots, k-1$. A well-known conjecture by Bonato, Janssen and Roshabin states that any connected graph on $n$ vertices has burning number at most $\lceil \sqrt{n} \rceil$. It was recently shown by Norin and Turcotte that the conjecture holds up to a factor of $1+o(1)$. In this note, we demonstrate how this result can be applied to determine the asymptotic value of the burning number for graph classes with given minimum degree. This is based on an observation about connected $2$-hop dominating sets, which may be of independent interest.
\end{abstract}

\section{Introduction}
Graph burning is a problem introduced by Bonato, Janssen and Roshanbin \cite{bonato15}. Given a graph $G$, we consider the following procedure. At each time step $t=1, 2, \dots$ we pick a vertex $v_t$ in $G$ and set it on fire. A vertex $v$ is burned at time $t$ if $v=v_t$, or either $v$ or one of its neighbors is burned at time $t-1$. The burning number of $G$, $b(G)$ is the smallest number of time steps to burn all vertices of $G$. Equivalently, the burning number is the smallest integer $k$ such that the vertex set of $G$ can be covered with balls of radii $0, 1, \dots, k-1$. The Graph Burning Conjecture, as first formulated by Bonato et al, states that any connected graph $G$ on $n$ vertices has burning number at most $\lceil \sqrt{n} \rceil$. This would be optimal as it can be observed (Theorem 2.9 \cite{bonato15}) that the path on $n$ vertices has burning number $\lceil \sqrt{n} \rceil$.

This conjecture has been confirmed in some special graph classes, such as caterpillars \cite{cat1,cat2}, spiders \cite{spid1, spid2}, and sufficiently large graphs with minimum degree at least $4$ \cite{pyrotechnics}, which was later improved to $3$ in \cite{NT22}. In the case of general connected graphs $G$ on $n$ vertices, it was shown in \cite{bonato15} that $b(G)\leq 2\lceil \sqrt{n}\rceil + 1 $, with gradually improving leading constant in \cite{burn1, burn2, pyrotechnics}, until recently when Norin and Turcotte \cite{NT22} showed that $b(G)\leq (1+o(1))\sqrt{n}$.

Given these results for general connected graphs, a natural way to prove stronger upper bounds on the burning number for some special class of graphs is to find a connected subgraph $H$ of $G$ with the property that any vertex in $G$ is close to $H$. For instance, the minimum size of a connected subgraph with the property that all vertices are at distance at most $1$ from it is known as the connected domination number $\gamma_c(G)$. This is equivalent to the maximum leaf spanning tree problem. For a connected graph with minimum degree at least $k$ for $k=3, 4$ or $5$, it is known \cite{tree3, tree45} that $\gamma_c(G)\leq 3n/(k+1) + O(1)$. Bastide et al. \cite{pyrotechnics} argued that $b(G) \leq b(H) + 1$, which combined with the improved bounds on the burning number of general connected graphs in \cite{NT22} implies that $b(G)\leq (1+o(1))\sqrt{3 n/(k+1)}$ for $k=3, 4$ or $5$. It is not too hard to see that these bounds are best possible for both $\gamma_c(G)$ and $b(G)$, as matching lower bounds can be obtained for any $k\geq 3$ by considering a ``$K_k-e$ necklace'' (see eg. Fig 1 in \cite{tree45}).

It is tempting to think the same statement holds for larger $k$, but this turns out to be false - at least for the connected domination number. An easy consequence of a result by Alon \cite{alon}, see the discussion on page 6 of  \cite{alonspencer}, is that there are $k$-regular graphs with $\gamma_c(G)\geq (1+o_k(1)) \frac{\ln k}k n$. In fact, this statement is even true if one drops the condition that $H$ needs to be connected. Caro, West and Yuster \cite{CWY} showed a matching upper bound of $\gamma_c(G)\leq (1+o_k(1)) \frac{\ln k}k n$ for any connected graph $G$ on $n$ vertices with minimum degree $k\geq 3$. 

In this note, we will give a short proof that the burning number of graphs with any minimum degree $k$ nevertheless behaves similar to the case of $k\leq 5$. We summarise this as follows.

\begin{theorem}\label{thm:mindeg} The burning number of any connected graph $G$ with $n$ vertices and minimum degree $k$ is at most $(1+o(1)) \sqrt{3n/(k+1)}$.
\end{theorem}

In addition to applying the result by Norin and Turcotte, our proof makes use of the following observation. If, instead of considering connected domination number, one allows vertices to be at distance at most two from the connected subgraph $H$, making $H$ a connected $2$-hop dominating set, then the $\log k$ factor does not appear for large $k$. To the authors knowledge, this appears to not have been observed before, but may be interesting on its own.

\begin{lemma}\label{lem:2hopdom}
Let $G$ be a connected graph with $n$ vertices and minimum degree $k$. There exists a connected subgraph $H$ with at most $3\lfloor n/(k+1) \rfloor -2$ vertices such that any vertex of $G$ is at distance at most $2$ from $H$.
\end{lemma}

Finally, by extending these ideas, we note that the Graph Burning Conjecture holds for all sufficiently large graphs with a weak degree condition.

\begin{theorem}\label{thm:weakdeg} For any $\varepsilon>0$ there exists $n_0=n_0(\varepsilon)$ such that any connected graph $G$ with at least $n_0$ vertices and where at least $\varepsilon n$ vertices have degrees distinct from $2$ satisfies $b(G)\leq (1-\Omega(\varepsilon))\sqrt{n}$.    
\end{theorem}

\section{Proof of Lemma \ref{lem:2hopdom} and Theorems \ref{thm:mindeg} and \ref{thm:weakdeg}}

\begin{proof}[Proof of Lemma \ref{lem:2hopdom}.] We construct the subgraph $H$ iteratively. Let $H_0$ be any vertex of $G$. For any $t=0, 1, \dots$ if there exists a vertex $v_t$ in $G$ at distance $3$ from $H_{t}$ then we let $H_{t+1}$ be the graph $H_t\cup P_t$ where $P_t$ is a path of length $3$ from $H_t$ to $v_t$. If there is no such vertex $v_t$, then we put $H=H_t$ and finish the process.

It only needs to be shown that $H$ produced in this way is not too big. To see this, observe that, on the one hand $H_t$, if it exists, contains exactly $1+3t$ vertices. On the other hand, consider the number of vertices $a_t$ in $G$ which are at distance at most $1$ from $H_t$. Clearly $a_0\geq k+1$. Moreover, if $H_{t+1}$ exists, then $a_{t+1} \geq a_t + k+1$ because $v_t$ and all of its neighbors are counted in $a_{t+1}$ but not in $a_t$. Since $a_t$ can never exceed $n$, it follows that the process terminates in a step $t\leq \lfloor n/(k+1) \rfloor -1$. Thus $H$ has at most $3\lfloor n/(k+1)\rfloor - 2$ vertices.
\end{proof}
\begin{proof}[Proof of Theorem \ref{thm:mindeg}.]
Let $H$ be as in Lemma \ref{lem:2hopdom}. Then $b(G)\leq b(H)+2$. By the result of Norin and Turcotte \cite{NT22}, the burning number of $H$ is at most $(1+o(1))\sqrt{3n/(k+1)}$, and the claimed result follows.
\end{proof}

\begin{proof}[Proof of Theorem \ref{thm:weakdeg}.]
Let $\varepsilon>0$, and let $G$ be a graph on $n$ vertices with at least $\varepsilon n$ vertices with degrees not equal to $2$. Below we will always assume $n$ is large in terms of $\varepsilon$. We will prove the theorem by showing that the connected domination number of $G$ is at most $(1-c\cdot \varepsilon)n + O(1)$ for some fixed $c>0$.

If $G$ has at least $\frac{\varepsilon}3 n$ vertices of degree $1$, then the statement is immediately clear as the set of all vertices with degree at least $2$ forms a connected dominating set of size at most $(1-\varepsilon/3)n$. Otherwise, $G$ has at most $\frac{\varepsilon}3 n$ vertices of degree $1$ and at least $\frac{2\varepsilon}3 n$ vertices of degree at least $3$. If so, we consider the following steps.

\begin{claim} Let $G'$ be the graph formed from $G$ by iteratively removing vertices of degree $1$ and replacing vertices of degree $2$ by an edge until the remaining graph has minimum degree $3$. Then $G'$ has at least $\frac{\varepsilon}3 n$ vertices.
\end{claim}
\begin{proof} Let us define the value of a graph as the number of vertices of degree at least $3$ minus the number of vertices of degree $1$. By assumptions above, $G$ has value at least $\frac{\varepsilon}3 n$, and it is easy to see that the contractions in the claim does not decrease the value of the graph.
\end{proof}

As $G'$ has minimum degree at least $3$, we have that $\gamma_c(G') \leq \frac34 |V(G')| + O(1) \leq  |V(G')|-\varepsilon n/12 + O(1).$


\begin{claim} Let $F$ be any connected graph. Let $F'$ be a graph formed from $F$ by removing a vertex of degree $1$ or replacing a vertex of degree $2$ by an edge. Then $\gamma_c(F) \leq \gamma_c (F')+1$.
\end{claim}
\begin{proof} Let $v$ be the vertex of $F$ removed to form $F'$, and let $U'$ be a minimum connected dominating set of $F'$. If $U'$ contains a neighbor of $v$, then we can form a connected dominating set of $F$ by taking $U=U'\cup \{v\}$. If no neighbors of $v$ are in $U'$ we instead take an arbitrary neighbor $w$ of $v$ and put $U=U'\cup \{w\}$.
\end{proof}

Combining these statements, it follows that $\gamma_c(G) \leq (1-\varepsilon/12) n + O(1)$, as desired. By the same argument as for Theorem \ref{thm:mindeg}, it follows that $b(G)\leq (1+o(1))\sqrt{(1-\varepsilon/12)n}$, which completes the proof.
\end{proof}

\section*{Acknowledgements} I thank Raphael Steiner for his feedback on simplifying the statement and proof of Lemma \ref{lem:2hopdom}.

\bibliographystyle{abbrv}
\bibliography{references}
\end{document}